\documentclass[11pt]{amsart}

\usepackage{esvect}
\usepackage{amssymb}
\usepackage{amsmath}
\usepackage{graphicx}
\usepackage{verbatim}
\usepackage{comment}
\usepackage[margin=0.75in]{geometry}
\usepackage{caption}
\usepackage{subcaption}
\usepackage{placeins}

\usepackage{bookmark}
\usepackage{hyperref}
\usepackage{graphicx}
\usepackage{float}
\usepackage{amsmath,amsthm,amssymb}
\usepackage{amsmath,amssymb}
\usepackage{epstopdf}
\usepackage{color}
\usepackage{wrapfig}
\usepackage{marvosym}
\usepackage{caption}
\usepackage{cleveref}
\usepackage{ragged2e}
\usepackage{mathtools}
\usepackage{verbatim}
\usepackage{amsthm}

\newcommand{\E}{\mathbb{E}}
\newcommand{\Z}{\mathbb{Z}}
\renewcommand{\P}{\mathbb{P}}

\newtheorem{thrm}{Theorem}

\newtheorem{lemma}{Lemma}
\newtheorem{cor}{Corollary}

\theoremstyle{remark}
\newtheorem{rem}{Remark}

\theoremstyle{definition}
\newtheorem{ex}{Example}

\title{Upper and Lower Bounds on the Speed of a One Dimensional Excited Random Walk }
\date{\today}
\author{Erin Bossen}
\address{Erin Bossen \\ Eastern Illinois University \\ Department of Mathematics and Computer Science \\ 600 Lincoln Avenue \\ Charleston, IL 61920 \\ USA }
\author{Brian Kidd}
\address{Brian Kidd \\  Purdue University \\ Department of Statistics \\ 250 N University St \\ West Lafayette, IN 47907 \\ USA}
\author{Owen Levin}
\address{Owen Levin \\ University of Minnesota \\ School of Mathematics \\ 206 Church Street S.E. \\ Minneapolis, MN 55455 \\ USA }
\author{Jonathon Peterson}
\address{Jonathon Peterson \\  Purdue University \\ Department of Mathematics \\ 150 N University St \\ West Lafayette, IN 47907 \\ USA }
\author{Jacob Smith}
\address{Jacob Smith \\ Franklin College \\ Department of Mathematics and Computing Sciences \\ 101 Branigin Blvd. \\ Franklin, IN  46131 \\ USA }
\author{Kevin M. Stangl}
\address{Keven Stangl \\ University of California, Los Angeles \\ Department of Mathematics \\ Los Angeles, CA 90024 \\ USA }


\begin{document}

\begin{abstract}
Excited random walks (ERWs) are a self-interacting non-Markovian random walk in which the future behavior of the walk is influenced by the number of times the walk has previously visited its current site. We study the speed of the walk, defined as $V = \lim_{n \rightarrow \infty} \frac{X_n}{n}$ where $X_n$ is the state of the walk at time $n$. While results exist that indicate when the speed is non-zero, there exists no explicit formula for the speed. 
 It is difficult to solve for the speed directly due to complex dependencies in the walk since the next step of the walker depends on how many times the walker has reached the current site. 
 We derive the first non-trivial upper and lower bounds for the speed of the walk.
In certain cases these upper and lower bounds are remarkably close together. 
\end{abstract}

\maketitle

\section{Introduction}

\indent A simple random walk on $\Z$ can be thought of as a simple discrete model for random motion where at each time step the 'walker' tosses a (possibly biased) coin and steps right if he gets a heads and left if he gets a tails. 
Mathematically, if we denote the position of the walk after $n$ steps by $S_n$ then we can represent the walk as $S_n = \sum_{i=0}^n \xi_i$ where the sequence of random variables $\xi_1,\xi_2,\xi_3,\ldots$ represent the successive steps of the walk. Since the steps are given by the outcomes of repeated tosses of a coin, the random variables $\{\xi_i\}_{i\geq 0}$ are independent and identically distributed (i.i.d.) with $P(\xi_1 = p)$ and $P(\xi_1 = -1) = 1-p$ (here $p \in (0,1)$ is the probability that the coin the walker is tossing comes up heads).

Simple random walks are very well known and much is known about them, but in this paper we will focus on a different model for random motion called an excited random walk.  
In an excited random walk, rather than the steps of the walk being i.i.d.\ the probability of the walker moving right ($+1$) or left ($-1$) from a site on the $n$-th step is a function of how many times the walker has stepped on that site by time $n$. 
To describe the excited random walk model, we begin by fixing an integer $M\geq 1$ and parameters $p_1, p_2, \dots, p_M \in (0,1)$. 
When the walker visits a location $i$ for the $j$th time, if $j\leq M$ then the walker tosses a coin with probability of heads $p_j$ while if $j>M$ the walker tosses a fair coin ($p=1/2$) to determine the next step left or right. 
That is, an excited random walk is a stochastic process $\{X_n\}_{n\geq 0}$ starting at $X_0 = 0$ and such that $X_{n+1} = X_n \pm 1$ and 
\[
 \P( X_{n+1} = X_n+1 \, | \, X_0 = x_0, X_1 = x_1, \ldots, X_n=x_n ) 
=
\begin{cases}
 p_j & \text{if } \#\{ k\leq n: x_k = x_n \} = j \leq M \\
 \frac{1}{2} & \text{if } \#\{ k\leq n: x_k = x_n \} > M. 
\end{cases}
\]

Excited random walks are sometimes also called ``cookie random walks'' due to the following interpretation of the dynamics. 
We imagine that initially there is an identical stack of $M$ cookies at each site. 
At every step the random walker takes the top cookie from the stack at the current site (if there is at least one cookie left) and eats it. 
The cookie induces an ``excitement'' or drift which causes the walker to step to the right with probability $p_j$ (or left with probability $1-p_j$). If the walker ever returns to a site where all the cookies have already been eaten then there is nothing to ``excite'' him and so he steps left/right with equal probability. 
Due to this ``cookie'' interpretation of excited random walks we will often refer to the parameter $M$ as the number of cookies at each site and the parameter $p_j$ as the ``strength'' of the $j$-th cookie.

\begin{figure}
\begin{center}\renewcommand{\arraystretch}{1.5}
\begin{tabular}{c|c}
 \includegraphics[width=0.4\textwidth, page=1]{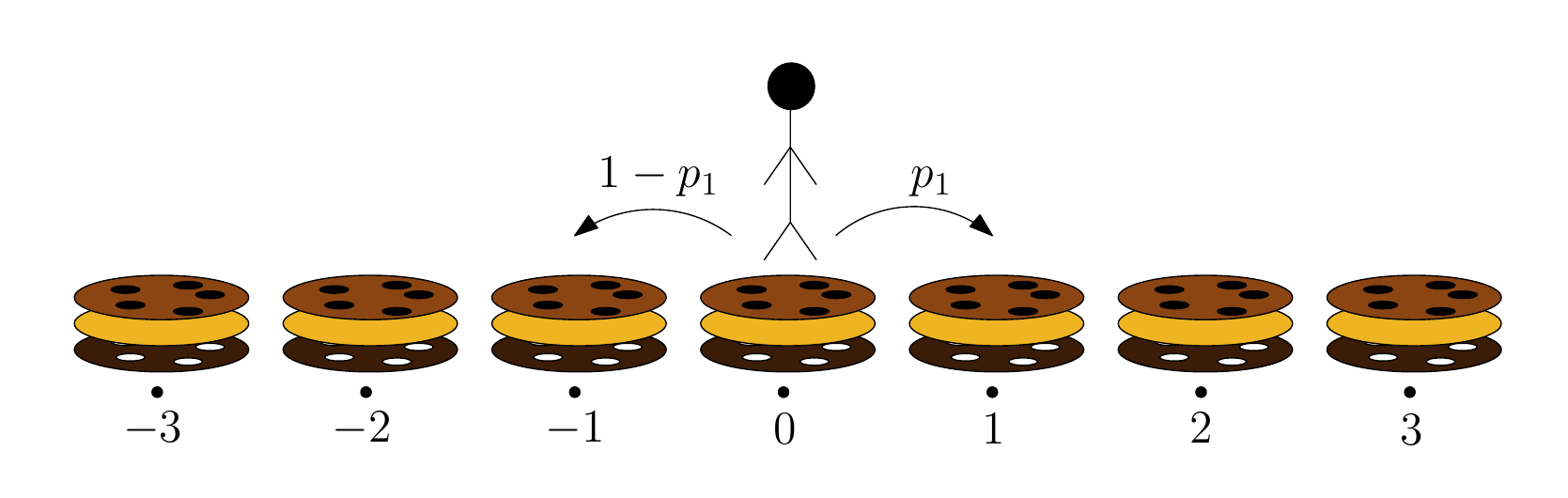}
&\includegraphics[width=0.4\textwidth, page=2]{CookieRW.pdf}\\ \hline
 \includegraphics[width=0.4\textwidth, page=3]{CookieRW.pdf} 
&\includegraphics[width=0.4\textwidth, page=4]{CookieRW.pdf} 
\end{tabular}
\end{center}
\caption{A partial example of an excited random walk with $M=3,\vec{p}=(p_1,p_2,p_3)$ and with transition probabilities shown -- Top left: the initial state of the walker --Top right: A possible state after 9 steps -- Bottom left: 10 steps into the same walk, with the most recent step to the right -- Bottom right: 11 steps into the walk, the walker is now in a state with no more cookies left and has equal transition probabilities to the left and right.}
\end{figure}

\subsection{Background and previous results}

Excited random walks were first introduced by Benjamini and Wilson in \cite{benjamini2003excited}. In the model considered by Benjamini and Wilson, however, there was only one cookie at each site $M=1$. This model was then generalized by Zerner in \cite{zerner2005multi} to allow for multiple cookies at each site, but with the restriction that all $p_j \geq \frac{1}{2}$; that is all cookies induced a non-negative drift for the walker. Kosygina and Zerner further generalized the model in \cite{kosyginazerner2008posneg} to allow for the possibility of both ``positive'' ($p_j > 1/2$) and ``negative'' ($p<1/2$) cookies in the stack of cookies at each site. 
In fact, the model of excited random walks is even more general than what we have described here. Certain results have even allowed for placing random cookie stacks at sites (rather than the same cookie stack at each site) and for infinitely many cookies at each site. 
In this paper, however, we will restrict ourselves to the simpler model described above of $M$ cookies at each site with strengths $p_1,p_2,\ldots,p_M$. 

The behavior of simple random walks is quite easy to analyze since as noted above the walk $S_n = \sum_{i=1}^n \xi_i$ is the sum of i.i.d.\ random variables. In particular, the law of large numbers implies that $\lim_{n\to\infty} \frac{S_n}{n} = E[\xi_1] = 2p-1$ with probability 1. That is, the random walk has a deterministic limiting speed of $2p-1$. Thus, if $p>1/2$ then the walk moves to the right and with positive speed while if $p<1/2$ the walk moves to the left with speed $1-2p$ (or equivalently, for any $p \in [0,1]$ the walker simply moves with \emph{velocity} $2p-1$). In either of these cases we say that the walk is \emph{transient} since it only visits any site a finite number of times. 
More generally, if a random walk is transient with non-zero speed, it is \emph{ballistic}.  For one-dimensional simple random walks, transience and ballisticity are equivalent, but as we will see in our discussion of excited random walks, this is not always the case.
The case $p=1/2$ is more delicate, but it was shown in 1921 by P\'olya \cite{polya1921rw} that a one-dimensional simple symmetric random walk is \emph{recurrent}; that is, the walk visits every site infinitely many times.

In contrast to simple random walks, the behavior of excited random walks is much more difficult to determine since the self-interacting nature of the walk creates dependencies among steps of the walk that are very hard to handle. Moreover, the behavior of the walk is at times like a biased random walk (on the first $M$ visits to sites) while at other times is like a symmetric random walk (after more than $M$ visits to a site). Thus, even the question of determining whether the excited random walk is recurrent or transient is quite difficult. 
In spite of these difficulties, a number of characteristics of excited random walks have been determined to depend on a single easy to calculate parameter. 
\begin{equation}\label{deltadef}
 \delta = \sum_{j=1}^M (2 p_j - 1). 
\end{equation}
We will use the notation $\delta_j = 2p_j-1$ for the drift of the $j$-th cookie in the cookie stack. Thus, the parameter $\delta = \sum_{j=1}^M \delta_j$ can be thought of as the net total drift contained in all the cookies in the cookie stack at each site.  

\begin{thrm}[Zerner \cite{zerner2005multi}, Kosygina \& Zerner \cite{kosyginazerner2008posneg}]\label{ERWtrans}
 The parameter $\delta$ determines the recurrence or transience of the excited random walk. 
\begin{enumerate}
 \item If $\delta>1$ then the walk is transient to the right; that is, $\P(\lim_{n\to\infty} X_n = +\infty) = 1$. 
 \item If $\delta<-1$ then the walk is transient to the left; that is, $\P(\lim_{n\to\infty} X_n = -\infty) = 1$.
 \item If $\delta \in [-1,1]$ then the walk is recurrent; that is, $\P(\liminf_{n\to\infty} X_n = -\infty, \, \limsup_{n\to\infty}X_n = +\infty) = 1$. 
\end{enumerate}
\end{thrm}

In \cite{zerner2005multi} Zerner also proved that excited random walks have a limiting speed. That is, given any parameters $M$ and $\vec{p} = (p_1,p_2,\ldots,p_M)$ for an excited random walk there is a constant $V_{M,\vec{p}} \in [-1,1]$ such that
\begin{equation}\label{LLN}
 \lim_{n\to\infty} \frac{X_n}{n}  = V_{M,\vec{p}}, \quad \text{with probability one.}
\end{equation}
Determining the exact value of the speed $V_{M,\vec{p}}$ as a function of $M$ and $\vec{p}$, however, remains an open problem and is the focus of this present paper. 
While there is still no explicit formula for $V_{M,\vec{p}}$ in general, it is known that the parameter $\delta$ determines exactly when the speed is positive, negative or zero. 

\begin{thrm}[Basdevant \& Singh \cite{basdevant2008}, Kosygina \& Zerner \cite{kosyginazerner2008posneg}]\label{speedpos}
 The parameter $\delta$ determines the sign of the limiting speed $V_{M,\vec{p}}$ of the excited random walk. 
\begin{enumerate}
 \item If $\delta > 2$ then $V_{M,\vec{p}} > 0$. 
 \item If $\delta < -2$ then $V_{M,\vec{p}} < 0$. 
 \item If $\delta \in [-2,2]$ then $V_{M,\vec{p}} = 0$. 
\end{enumerate}
\end{thrm}

\begin{rem}
 Note that Theorems \ref{ERWtrans} and \ref{speedpos} together highlight a very peculiar feature of excited random walks: If $\delta \in (1,2]$ then the walk is transient to the right, but with zero asymptotic speed. 
At first this might seem contradictory, but in fact it holds because in this case $X_n$ grows to infinity roughly like $n^{\delta/2}$ if $\delta \in (1,2)$ or like $n/\log n$ if $\delta =2$ \cite{basdevant2008growth,kosyginazerner2008posneg}. 
\end{rem}

\begin{ex} 
Let $M=3$ and $\vec{p} = (p,p,p)$. Then $\delta = 6p-3$. 
\begin{enumerate}
 \item If $p \in [\frac{1}{3},\frac{2}{3}]$ then $\delta\in[-1,1]$, so the walk is recurrent.
 \item If $p \in [\frac{1}{6},\frac{5}{6}]$ then $\delta\in[-2,2]$, so the walk is transient with $V_{M,\vec{p}}=0$.
 \item If $p \in [0,\frac{1}{6})$ then $\delta < -2$, so the walk is ballistic with $V_{M,\vec{p}} < 0$.
 \item If $p \in (\frac{5}{6},1]$ then $\delta > 2$, so the walk is ballistic with $V_{M,\vec{p}} > 0$.
\end{enumerate}
\end{ex}

 \begin{rem}
 It should be noted that if $p_i\in(0,1)$ for all $i$, then unless $M \geq 3$, $V_{M,\vec{p}}=0$. If $M < 3$ then $\delta < 4\cdot1-2=2.$ Thus, $V_{M,\vec{p}}$ is non-positive. A symmetric argument shows that $\delta >-2$ and thus $V_{M,\vec{p}}=0$ unless $M\geq 3$
\end{rem}

Theorem \ref{speedpos} shows that we can identify the speed of the excited random walk exactly when the speed is zero (when $\delta \in [-2,2]$). 
However, as noted above when the speed is non-zero (when $\delta \notin [-2,2]$) then there is no explicit formula for the speed $V_{M,\vec{p}}$. 
The focus of this paper is to compute explicit upper and lower bounds for the speed in these cases. 
For simplicity we will restrict ourselves to the case of positive speed ($\delta > 2$) since the negative speed case can by handled similarly by symmetric arguments. 
Prior to this paper, when $\delta>2$ the only known upper and lower bounds on the speed were the trivial ones
\[
 0 < V_{M,\vec{p}} \leq \max_{j\leq M} (2 p_j - 1). 
\]
The upper bound on the right is the speed of a simple random walk which moves to the right with probability $p^* = \max_{j\leq M} p_j$ on each step. Since this simple random walk is always at least as likely to step right as the excited random walk, it is easy to see that the excited random walk has a speed that is less than or equal to that of this simple random walk.
We will develop a method below for obtaining much better bounds than these trivial bounds. In particular, in the case of $M = 3$ cookies per site we will obtain upper and lower bounds which differ by at most 0.0194565. 

The rest of the paper will be organized as follows. We begin with a brief introduction to the theory of Markov chains to cover results we will use.  Then we describe a particular Markov chain related to excited random walks, known as the backward branching process.  We discuss known results about this Markov chain and how they relate to the speed of an excited random walk. Afterward, we derive  bounds  on the speed using properties of the backward branching process.  We end with a discussion of how well these bounds approximate the speed.

\section{A related Markov chain}
In this section we will introduce a Markov chain that is useful for studying the speed of excited random walks. 
First, however, we will give a short overview of the notation and terminology of Markov chains and recall a few useful facts about Markov chains.

\subsection{Markov Chains}

Recall that a Markov chain on a countable state space $I$ is a stochastic process $\{Z_n\}_{n\geq 0}$ such that 
for any choice of $n\geq 1$ and  $i_0,i_1,\ldots,i_n,i_{n+1} \in I$ we have 
\begin{align*}
  \P(Z_{n+1} = i_{n+1} \, | \, Z_0 = i_0, \, Z_1 = i_1, \ldots, Z_{n-1} = i_{n-1}, \,Z_n= i_n) 
&= \P(Z_{n+1} = i_{n+1} \, | \, Z_n = i_n ) \\
&= \P(Z_1 = i_{n+1} \, | \, Z_0= i_n ).
\end{align*}
The transition matrix for the Markov chain is the matrix
\[
 P = (p(i,j))_{i,j \in I}, \quad \text{where } p(i,j) = \P(X_1=j \, | \, X_0=i). 
\]
For ease of notation, if the Markov chain starts at $Z_0 = i$ we will write $\P_i(\cdot)$ in place of $\P(\cdot \, | \, Z_0=i)$. 
If the Markov chain starts from a random initial condition given by $\mu = (\mu(i))_{i\in I}$ where $\mu(i)$ is the probability that the Markov chain starts at $Z_0=i$, then we will denote this with the notation $\P_\mu$; that is, $\P_\mu(\cdot)= \sum_i \mu(i) \P_i(\cdot)$. Expectations with respect to the probability distributions $\P_i$ or $\P_\mu$ for the Markov chain are denoted by $\E_i$ or $\E_\mu$, respectively.

A special choice of an initial distribution is what is called a \emph{stationary distribution}. A probability distribution $\pi = (\pi(i))_{i \in I}$ is a stationary distribution for the Markov chain $Z=\{Z_n\}_{n\geq 0}$ if $\P_\pi(Z_1=j) = \P_\pi(Z_0=j) = \pi(j)$ for all $j \in I$; that is, if $Z_1$ has the same distribution $\pi$ as $Z_0$ (and thus, by induction, $Z_n$ has the same distribution as $Z_0$ for all $n\geq 1$). 
If $\pi$ is a stationary distribution then 
\[
 \pi(j) = \P_\pi(Z_1=j) = \sum_{i \in I} \pi(i) \P_i(X_1=j) = \sum_{i \in I} \pi(i) p(i,j),
\]
so that viewing $\pi=(\pi(i))_{i\in I}$ as a row vector we have $\pi = \pi P$; that is, $\pi$ is a left eigenvector of the transition matrix $P$ with eigenvalue $1$. 
If the state space $I$ of the Markov chain is finite, then computing the stationary distributions is a simple problem in linear algebra. However, if the state space $I$ is countably infinite then computing stationary distributions is much more difficult and in fact, for some inifinite state Markov chains there are no stationary distributions. It is known, however, that if the Markov chain is irreducible (that is, if it is possible starting at any state $i$ to eventually reach any other state $j$) and there is a stationary distribution then it is unique.  

Stationary distributions are important for the analysis of Markov chains because they can be used to determine the long run asymptotics of the Markov chain. 
For instance, if the Markov chain is irreducible and a stationary distribution $\pi$ exists, then it is known that for any initial starting condition that 
\[
 \lim_{n\to\infty} \frac{1}{n} \sum_{k=1}^n Z_k = \E_\pi[Z_0] = \sum_{j\in j} \pi(j) j, \quad \text{with probability one.}
\]

\subsection{The Backward Branching Process}

Because the transition probabilities of the excited random walk depend on the number of prior visits to the present location and not only on the current location of the walk, an excited random walk is not a Markov chain. 
However, there is a Markov chain that we can study that can give information about the excited random walk. This Markov chain is often referred to in the literature as the ``backward branching process'' due to some structural similarity with models for population growth known as branching processes. 
The backward branching process is related to the excited random walk through an analysis of the number of left (or backward) crossings of edges of the excited random walk before the walk reaches some point to the right for the first time. We refer the reader interested in the details of this connection to \cite{basdevant2008}. Here just provides a description of the transition probabilities for this Markov chain and the relevance to the limiting speed of the excited random walk.

To describe the transition probabilities for the backwards branching process, we imagine an infinite sequence of independent coin flips where for the first $M$ flips we use coins which come up heads with probability $p_j$ for $j=1,2,\ldots,M$ and then for all subsequent flips we use a fair coin. Mathematically we can represent this as the sequence $\{\xi_j\}_{j\geq 1}$ of independent Bernoulli random variables where 
\[
 \P(\xi_j = 1) = 
\begin{cases}
 p_j &\text{if } j\leq M \\
 \frac{1}{2} &\text{if } j > M. 
\end{cases}
\]
Next, for any $m\geq 1$ we let 
\[
 F_m = \inf\left\{ k\geq 0: \sum_{j=1}^{m+k} \xi_j \right\}. 
\]
Again viewing the $\{\xi_j\}_{j\geq 1}$ as the outcome of successive coin tosses we have that $F_m$ can be interpreted as the number of ``tails'' before the $m$-th ``heads.'' 
Finally, using this notation we are able to define the backward branching process associated to the excited random walk with parameters $M$ and $\vec{p}=(p_1,p_2,\ldots,p_M)$ as the Markov chain $Z=\{Z_n\}_{n\geq 0}$ on $\Z_+ = \{0,1,2,\ldots\}$ with transition probabilities given by
\[
 p(i,j) = \P( F_{i+1} = j ), \quad \text{for } i,j\geq 0. 
\]

\begin{ex}\label{ex:tp}
Some transition probabilities which we will use later in Lemma~\ref{pi0bounds} are given below.  Also we show the full transition matrix for when $p_1=p_2=p_3=p$. When $M=3$ cookies per site we have 
 \begin{itemize}
\item $p(0,0) = p_1$ \hfill(no tails before a single heads)
\item $p(0,1) =(1- p_1)p_2$ \hfill(one tail before a single heads)
\item $p(0,2) = (1-p_1)(1-p_2)p_3$ \hfill(two tails before a single heads)
\item $p(0,k) = (1-p_1)(1-p_2)(1-p_3)/2^{k-2}$ \quad for $k\geq 3$ \hfill($k$ tails before a single heads)
\item $p(1,0) =p_1p_2$ \hfill(no tails before two heads)
\item $p(1,1) =(1-p_1)p_2p_3+p_1(1-p_2)p_3$ \hfill(one tail before two heads)
\item $p(k,0) = p_1p_2p_3/2^{k-2}$ for $k>3$ \hfill(no tails before $k+1$ heads)
 \end{itemize}
 
 In the $M=3$ case where $p_1=p_2=p_3=p,$ (letting $q:=1-p$), 
the initial entries of the transition matrix (with $i,j \leq 2$) are
\[
\left(\begin{array}{cccc}
    p       & pq & pq^2 &  \\
    p^2      & 2p^2q & \frac{3}{2}pq^2 & \cdots \\
    p^3      & \frac{3}{2}p^2q & \frac{3}{4}(pq^2+ p^2q) & \\
& \vdots && \ddots
\end{array}\right)
\]
and the remaining entries (when either $i$ or $j>2$) are given by 
\[
p(i,j) = \frac{1}{2^{i+j-2}}\left[\binom{i+j-3}{i-3}p^3 + \binom{i+j-3}{j-3}q^3 + 3\binom{i+j-3}{i-2}p^2q + 3\binom{i+j-3}{j-2}pq^2\right]
\]
\end{ex}

The Markov chain $Z$ was first introduced in the study of excited random walks by Basdevant and Singh in \cite{basdevant2008}. 
It is easy to see that the Markov chain $Z$ is irreducible since $p(i,j)>0$ for all $i,j \geq 0$. Moreover, Basdevant and Singh showed that the Markov chain $Z$ has a (unique) stationary distribution $\pi$ whenever the parameter $\delta > 1$ (or equivalently, by Theorem \ref{ERWtrans}, when the excited random walk is transient to the right).  
Most importantly, Basdevant and Singh proved that the limiting speed $V_{M,\vec{p}}$ for the excited random walk can be expressed in terms of the stationary distribution for the Markov process $Z$ in the following theorem.

\begin{thrm}[Basdevant \& Singh \cite{basdevant2008}]\label{speed} 
Suppose the parameters $M$ and $\vec{p} = (p_1,p_2,\ldots,p_M)$ are such that the speed $V_{M,\vec{p}} > 0$ (that is, $\delta>2$). 
If $\pi$ is the stationary distribution for the corresponding backward branching process $Z=\{Z_n\}_{n\geq 0}$, then 
\begin{equation}\label{VEZ}
	V_{M,\vec{p}}=\frac{1}{1+2\E_\pi[Z_0]}.
\end{equation}
\end{thrm}

A rationalization for and proof sketch of Theorem~\ref{speed} comes from the following. 
Because $\delta>2$ the walk $X$ is transient and almost surely $\lim_{n\to\infty}\frac{X_n}{n}=V_{M,\vec{p}}>0$.  
In such situations, it holds that almost surely \[\lim_{n\to\infty}\frac{X_n}{n}=\frac{1}{\lim_{n\to\infty}\frac{T_n}{n}}\]
where $T_n$ is the hitting time of site $n$. 
Essentially, this identity is just noting that distance over time can be expressed in terms of two different quantities for $X$ and each are equivalent to the velocity of the walk. 

Now, the hitting time limit can be expressed in terms of the backward branching process by $\lim_{n\to\infty}\frac{T_n}{n}=\lim_{n\to\infty}\frac{n+2\sum_{k=1}^n Z_k}{n}.$
To see this, we count the number of steps making up the hitting time to site $n$. The number of total steps down from positive site $k$ to site $k-1$ before the walk reaches $n$ is $\sum_{k=1}^n Z_k$.  Each of these down steps is cancelled by one step back up to to site $k$ before reaching $n$.  In addition, we have the final up step from each positive site $k$ up to $n$ which is $n$ steps.  Lastly, $T_n$ contains the total number of steps from $0$ to $-1$ and all the steps contained in the negative half line.  Because $X$ is transient to $+\infty$ when $\delta>2$, there are a finite (random) number $L$ of these steps and the limit of $\frac{L}{n}\to0$ almost surely as $n$ goes to $\infty$.
Then we have the following equalities which imply the conclusion of Theorem~\ref{speed}.
\[\frac{1}{V_{M,p}}
=\lim_{n\to\infty}\frac{T_n}{n}
=\lim_{n\to\infty}\frac{L+n+2\sum_{k=1}^n Z_k}{n}
=\lim_{n\to\infty}\frac{L}{n}+\frac{n}{n}+2\frac{1}{n}\sum_{k=1}^n Z_k
=1+2\E_\pi[Z_0]\]

While Theorem \ref{speed} expresses the speed $V_{M,\vec{p}}$ in terms of the stationary distribution of the backward branching process, unfortunately, this doesn't give an explicit formula for the speed since there is not yet an explicit formula for the stationary distribution $\pi$ (solving the infinite system of equations $\pi P = \pi$ is too difficult).
In the following section, however, we will develop some methods which can be used to obtain rigorous upper and lower bounds on $\E_\pi[Z_0]$ and consequently upper and lower bounds on $V_{M,\vec{p}}$.

\section{Reduction of the formula for the speed}
In this section we will show how some recursive formulas for the probability generating function of the distribution $\pi$ can be used to get useful approximations (upper and lower bounds) of $\E_\pi[Z_0]$. The starting point of our analysis of the speed of the excited random walk is a recursive formula for the probability generating function $G(s) := \sum_{k=0}^\infty \pi(k) s^k$ of the stationary distribution $\pi$ for the Markov chain $Z$. 
Basdevant and Singh \cite{basdevant2008} showed that $G(s)$ is the unique solution of the functional equation 
\begin{equation}\label{Grec}
1 - G\left(\frac{1}{2-s}\right) = A(s) [1-G(s)]+ B(s), \qquad s \in [0,1]
\end{equation}
 where
\begin{equation*}
 A(s) = \frac{1}{(2-s)^{M-1}\E_{M-1}[s^{Z_1}]},  \\
\end{equation*}
and
\begin{equation}\label{Beq}
 B(s) = 1 - \frac{1}{(2-s)^{M-1}\E_{M-1}[s^{Z_1}]} + \sum_{k=0}^{M-2}\pi(k)\left(\frac{\E_k[s^{Z_1}]}{(2-s)^{M-1}\E_{M-1}[s^{Z_1}]} - \frac{1}{(2-s)^k} \right)
\end{equation}
While the recursive equation \eqref{Grec} is still to hard to solve explicitly, using the fact that $\frac{1}{2-s} \approx s$ when $s \approx 1$ Basdevant and Singh were able to use \eqref{Grec} to obtain asymptotics of the function $G(s)$ near $s=1$. This is particularly useful because of the property of probability generating functions that 
\begin{equation}\label{Gprime}
 G'(1) = \sum_{k=1}^\infty \pi(k) k = \E_\pi[Z_0]. 
\end{equation}
By careful analysis of this recursive equation near $s=1$ and using the formula \eqref{VEZ} for the speed, Basdevant and Singh were able to deduce the following implicit formula for the speed of an ERW. 
 \begin{thrm}\label{speed2}[Basdevant \& Singh \cite{basdevant2008}]\label{speedB} When the speed is nonzero (i.e. when $\delta > 2)$, then $\E_\pi[Z_0] = G'(1) = \frac{B''(1)}{2(\delta-2)}$ and consequently the speed is is equal to 
\begin{equation}\label{speedform2}
V_{M,\vec{p}} = \frac{\delta - 2}{\delta - 2 + B''(1)},  
\end{equation}
where $B(s)$ is defined in \eqref{Beq}. 
 \end{thrm}

In deriving the representation \eqref{speedform2} for the speed, Basdevant and Singh were primarily interested in determining when the speed $V_{M,p}$ was positive. However, an additional consequence of this formula is that it comes much closer to giving an explicit formula for the speed. 
While computing $\E_\pi[Z_0]$ using the standard formula in \eqref{Gprime} requires knowing all of the stationary distribution, Theorem \ref{speedB} shows we can instead compute this using only the $M-1$ values $\pi(0),\pi(1),\ldots,\pi(M-2)$. 
This is because all of the probability generating functions $\E_k[ s^{Z_1}]$ can be computed explicitly so that the only unknown terms in $B(s)$ are $\pi(0),\pi(1),\ldots,\pi(M-2)$.

\begin{ex}\label{ex:pgf}
In the general case of $M=3$ cookies, the formula for $B(s)$ involves $\E_k[s^{Z_1}]$ for $k=0,1,2$. These can be explicitly computed using the formulas for the transition probabilities $p(k,j)$ for the backward branching process.
\begin{align*}
 \E_0[ s^{Z_1}] &=p(0,0)+sp(0,1)+s^2p(0,2)+\sum_{k=3}^\infty s^k p(0,k)\\
 	&=p_1 +s [(1 - p_1)p_2 ]+s^2[(1 - p_1)(1 - p_2)p_3]+ (1 - p_1) (1 - p_2) (1 - p_3)\sum _{k=3}^{\infty } \frac{s^k}{2^{k-2}}\\
 	&=p_1+s[(1-p_1) p_2 ]+s^2[(1-p_1) (1-p_2) p_3]-\frac{(1-p_1) (1-p_2) (1-p_3) s^3}{s-2}
\end{align*}
Similar explicit calculations show that 
\begin{align*}
 \E_1[ s^{Z_1}] 
 &=\frac{s (2 p_2 (s-1)-s) (2 p_3 (s-1)-s)-p_1 (s-1) \left(p_2 \left(2 p_3 (3 s-4) s-3 s^2+4\right)+2 s (s-2 p_3 (s-1))\right)}{(s-2)^2},
\end{align*}
and 
\begin{align*}
 \E_2[ s^{Z_1}] &= \frac{(2 p_1 (s-1)-s) (2 p_2 (s-1)-s) (2 p_3 (s-1)-s)}{(s-2)^3}
\end{align*}
\end{ex}

As noted above, Theorem \ref{speed2} shows that the speed $V_{M,\vec{p}}$ for an excited random walk can be expressed in terms of only the unknown values $\pi(0),\pi(1),\ldots,\pi(M-2)$. 
The following lemma, however, gives a linear relation among these parameters so that we can actually eliminate one of the unknowns.

\begin{lemma}
\label{genabc} The unique stationary distribution $\pi$ of $\{Z_n\}_{n\geq 0}$ satisfies the following equation.
$$\delta - 1 = \sum_{k=0}^{M-2} \pi(k) \left(\E_k[Z_1] - k - 1 + \delta \right)$$
\end{lemma}
\begin{rem}
 Note that for any fixed excited random walk parameters, $M$ and $\vec{p}$, the expectations $\E_k[Z_1] = \sum_{j=0}^{\infty} j p(k,j)$ appearing in Lemma \ref{genabc} can be explicitly calculated. 
\end{rem}

 \begin{proof}
 Due to properties of the stationary distribution we know:
 \begin{align*}
 \E_\pi[Z_0] &= \E_\pi[Z_1] 
 \end{align*}
or equivalently
\begin{equation}\label{EZ0EZ1}
 \sum_{k=0}^\infty k\pi(k) = \sum_{k=0}^\infty \pi(k)\E_k[Z_1]
\end{equation}
In general, the expectations $\E_k[Z_1]$ have to be calculated individually using the transition probabilities for the Markov chain $\{Z_n\}_{n\geq 0}$. However, Basdevant and Singh showed in \cite[Lemma 3.3]{basdevant2008}
that the following pattern emerges when $k\geq M-1$. 
\begin{equation}\label{EkZ}
 \E_k[Z_1] = k+1-\delta, \quad \forall k\geq M-1. 
\end{equation}
(We provide a proof of \eqref{EkZ} in the Appendix.)
Using this, and splitting both sums in \eqref{EZ0EZ1} into $k\leq M-2$ and $k\geq M-1$ we obtain 
\[
 \sum_{k=0}^{M-2} k\pi(k) +\sum_{k=M-1}^{\infty} k\pi(k) = \sum_{k=0}^{M-2} \pi(k)\E_k[Z_1] + \sum_{k=M-1}^{\infty} (k+1-\delta)\pi(k)
\]
Noting that $\sum_{k=M-1}^{\infty} k\pi(k)$ appears on both sides, we reduce this to 
\begin{align*}
 \sum_{k=0}^{M-2} k\pi(k) &= \sum_{k=0}^{M-2} \pi(k)\E_k[Z_1] +  (1-\delta)\sum_{k=M-1}^{\infty}\pi(k) \\
&= \sum_{k=0}^{M-2} \pi(k)\E_k[Z_1] +  (1-\delta) - (1-\delta)\sum_{k=0}^{M-2} \pi(k)
\end{align*}
where in the last equality we used that $\sum_{k=M-1}^{\infty}\pi(k) = 1-\sum_{k=0}^{M-2} \pi(k)$ because $\pi$ is a probability distribution.
The statement of the lemma is then obtained by simplifying. 
\end{proof}


As a special case, when there are $M=3$ cookies Lemma \ref{genabc} gives a simple linear relation between $\pi(0)$ and $\pi(1)$. 
 \begin{cor}
 \label{abc}
     For $M=3$ cookies with strength $\vec{p} = (p_1,p_2,p_3)$, the following linear equation follows from above: $$a\pi(0) + b\pi(1) = c$$ where (recalling the notation $\delta_j = 2p_j-1$) we have 
     \begin{align*}
     a &:= p_1(\delta_2 + \delta_3) + p_2\delta_3(1-p_1) \\
     b &:= \delta_3p_1p_2 \\
     c &:= \delta - 1
     \end{align*}
 \end{cor}
\begin{proof}
 When $M=3$, the equation in Lemma \ref{genabc} becomes 
 %
 \begin{align}
     \label{abc-unsubbed}
     \delta - 1 = [\E_0[Z_1]+\delta-1]\cdot \pi(0)    +    [\E_1[Z_1]+\delta-2]\cdot \pi(1)
 \end{align}
 Next, note that $E_0[Z_1]$ and $E_1[Z_1]$ can be explicitly calculated from the known transition  probabilities for $Z$ (compare with Examples \ref{ex:tp} and \ref{ex:pgf} above). For example,
\begin{align*}
    \E_0[Z_1]   &= 0(p_1) + 1(1-p_1)p_2 + 2(1-p_1)(1-p_2)p_3 + (1-p_1)(1-p_2)(1-p_3) \sum_{k=3}^{\infty}\frac{k}{2^{k-2}}\\ 
                &= (1-p_1)p_2 + 2(1-p_1)(1-p_2)p_3 + 4(1-p_1)(1-p_2)(1-p_3)\\
                &= 4-4p_1-3p_2-2p_3+3p_1p_2+2p_1p_3+2p_2p_3-2p_1p_2p_3,
\end{align*}   
and similarly it can be shown that  
\begin{align*}
     \E_1[Z_1] 
       &= 5-2(p_1+p_2+p_3)-p_1 p_2(2p_3-1) = 2-\delta-p_1p_2\delta_3.
 \end{align*}
Substituting these formulas for $\E_0[Z_1]$ and $\E_1[Z_1]$ into (\ref{abc-unsubbed}) and simplifying we obtain the statement of the corollary.
\end{proof}

\section{Bounds on the Speed}

Theorem \ref{speed2} and Lemma \ref{genabc} combined show that the speed $V_{M,\vec{p}}$ of an excited random walk with $\delta > 2$ can be computed in terms of only the unknown values $\pi(0),\pi(1),\ldots,\pi(M-3)$.
Actually computing this function, however, is rather involved as especially computing the $B''(1)$ is a tedious task. Thus, for the remainder of the paper we will restrict ourselves to the case $M=3$ so that explicit computations can be done. 
With the aid of Mathematica to compute the derivatives in $B''(1)$, we were able to show the following. 
\begin{thrm}\label{Vpi0}
For an excited random walk with $M=3$ cookies of strengths $\vec{p} = (p_1,p_2,p_3)$, if $\delta>2$ the limiting speed is equal to 
 \begin{equation}\label{Vpiform}
V_{3,\vec{p}} = \frac{f_1}{f_2+f_3\cdot\pi(0)},
\end{equation}
where 
\begin{align*}
f_1 &= 2 p_1 + 2 p_2 + 2 p_3 - 5\\
f_2 &= 9 + 8(p_1p_2 + p_1p_3+p_2p_3) -10(p_1+p_2+p_3)\\
f_3 &= 2(2p_3-1)(p_1+p_2-3p_1p_2)
\end{align*}
\end{thrm}

The formula in equation \eqref{Vpiform} doesn't quite calculate $V_{3,\vec{p}}$ explicitly since we do not known the value of $\pi(0)$. However, the following lemma shows that we can easily use this formula to compute upper and lower bounds on the speed. 

\begin{lemma}\label{monotone}
Let $f_1,f_2$ and $f_3$ be as in Theorem \ref{Vpi0}. Then, if $\delta = \sum_{j=1}^3 (2p_j-1) > 2$ the function $x \mapsto \frac{f_1}{f_2+ f_3 x}$ is strictly positive and increasing for $x \in [0,1]$. 
\end{lemma}
\begin{proof}
 If $g(x) = \frac{f_1}{f_2+ f_3 x}$, then $g'(x) = \frac{-f_1 f_3}{(f_2+f_3 x)^2}$.
Thus, to show that $g(x)$ is decreasing we need only to show that $f_1 f_3 < 0$ when $p_1,p_2,p_3$ are such that $\delta > 2$. 
Note first of all that $\delta > 2$ is equivalent to $p_1+p_2+p_3 > \frac{5}{2}$. Therefore, 
\[
 f_1 =2(p_1+p_2+p_3) - 5 > 0, 
\]
and so it remains to show $f_3 < 0$. 
To see this, note that since $p_1,p_2$ and $p_3$ are each at most one then the condition $\delta>2$ implies that they are all strictly larger than $1/2$. Thus, $f_3=2(2p_3-1)(p_1+p_2-3p_1p_2)<0$ if $p_1+p_2-3p_1p_2 < 0$.  
When $\delta>2$, it follows that $p_1+p_2 \in (3/2,2)$. Therefore, if we fix $t \in (3/2,2)$ and if $p_1+p_2 = t$ then $p_1+p_2 - 3p_1p_2 = t-3 p_1(t-p_1) = 3p_1^2+(1-3p_1)t$ and we wish to show that this is negative for all $p_1 \in [t-1,1]$. However, since $ 3p_1^2+(1-3p_1)t$ is convex in $p_1$ we need only to check the value at the endpoints $p_1=t-1$ and $p_1=1$, and at both endpoints this evaluates to $3-2t < 0$. 
This completes the proof that $f_3 < 0$ whenever $\delta>2$ and thus also that $g(x)$ is decreasing for $x\in [0,1]$.  
Since we have already shown that $f_1>0$ and $f_3<0$ when $\delta>2$, it will follow that $g(x)$ is non-negative on $[0,1]$ if we can show that $f_2+f_3>0$ whenever $\delta > 2$. 
This will be accomplished by showing that
\begin{equation}\label{sumpos}
 f_2+f_3\geq 0 \quad\text{when } \delta = 2,
\end{equation}
and
\begin{equation}\label{partials}
 \frac{\partial}{\partial p_i} (f_2+f_3)>0, \quad \text{for } i=1,2,3 \text{ whenever } \delta >2.
\end{equation}

To show \eqref{sumpos}, note that if $\delta=2$ then $p_1+p_2+p_3 = \frac{5}{2}$. Therefore, substituting $p_3 = \frac{5}{2} - p_1-p_2$ into $f_2+f_3$ and then factoring we have
\begin{align*}
 (f_2+f_3)(p_1,p_2,\tfrac{5}{2}-p_1-p_2) 
&=-16+28p_1-12p_1^2+28p_2-40p_1p_2+12p_1^2p_2-12p_2^2+12p_1p_2^2 \\
&= 4(1-p_1)(1-p_2)(3p_1+3p_2-4). 
\end{align*}
However, if $\delta=2$ then $p_1+p_2 = \frac{5}{2}-p_3 \geq \frac{3}{2}$ and thus $3p_1+3p_2 - 4 \geq \frac{9}{2}-4 = \frac{1}{2}$. 
From this, the claim in \eqref{sumpos} follows. 

To show \eqref{partials}, note that direct computation of derivatives yields
\begin{align*}
 \frac{\partial(f_2+f_3)}{\partial p_1} &=-12+14p_2+12p_3-12p_2p_3 = 2p_2 -12(1-p_2)(1-p_3) \\
 \frac{\partial(f_2+f_3)}{\partial p_2} &=-12+14p_1+12p_3-12p_1p_3 = 2p_1 -12(1-p_1)(1-p_3) \\
 \frac{\partial(f_2+f_3)}{\partial p_3} &= -10 + 12p_1+12p_2-12p_1p_2 = 2-12(1-p_1)(1-p_2).
\end{align*}
For the partial derivative with respect to $p_1$, $\delta > 2$ implies that $p_3 > \frac{3}{2} - p_2$ so that 
\[
 (1-p_2)(1-p_3)<(1-p_2)(p_2-1/2) \leq \frac{1}{16}. 
\]
Also, since $\delta>2$ implies $p_2>1/2$ then we have that $\frac{\partial(f_2+f_3)}{\partial p_1} > 2(1/2)-12(1/16) = 1/4>0$. 
Similar arguments show that $\frac{\partial(f_2+f_3)}{\partial p_2} > 1/4$ and  $\frac{\partial(f_2+f_3)}{\partial p_3} > 5/4$ when $\delta > 2$. 
This completes the proof of \eqref{partials} and thus also the proof of the lemma. 
\end{proof}

Using Lemma \ref{monotone}, it follows that we can obtain upper and lower bounds on $V_{3,\vec{p}}$ by using the simple bounds $0 \leq \pi(0) \leq 1$; that is $\frac{f_1}{f_2} \leq V_{3,\vec{p}} \leq \frac{f_1}{f_2+f_3}$. 
However, we can get improved upper bounds on $\pi(0)$ by using the fact that $\pi$ is not just a probability distribution but also a stationary distribution for the Markov chain $\{Z_n\}_{n\geq 0}$.

\begin{lemma}\label{pi0bounds}
For an excited random walk with $M=3$ cookies of strengths $\vec{p} = (p_1,p_2,p_3)$, 
\[
\frac{c\cdot p_1p_2}{b\cdot (1-p_1)+a\cdot p_1p_2} \leq \pi(0) \leq \frac{c}{\frac{b\cdot (1-p_1)p_2}{1-((1-p_1)p_2p_3+p_1(1-p_2)p_3)}+a},
\]
where $a, b$, and $c$ are defined in Corollary \ref{abc}.
 \end{lemma}
 \begin{proof}
Since $\pi$ is the stationary distribution of a Markov chain with transition probability matrix $P = (p(i,j))_{i,j\geq 0}$, then we know that the (infinite) matrix equation $\pi = \pi P$ holds. 
That is, 
\[ 
 \pi(i) = \sum_{k=0}^\infty \pi(k)p(k,i), \quad \text{for any } i\geq 0.
\]
If we drop all but the first two terms in the sum on the right we then obtain the inequality
 \begin{equation}
 \pi(i)    
\geq \pi(0)p(0,i)+\pi(1)p(1,i) \label{piiLowerBound}
 \end{equation}
 where $p(i,j)$ is the transition probability from state $i$ to state $j$ in the backward branching process.
 %
 For a lower bound on $\pi(0)$ we use $i=0$ in \eqref{piiLowerBound} and then Corollary \ref{abc} to get 
 \begin{align*}
 \pi(0)&\ge  p(0,0)\pi(0)+ p(1,0)\pi(1)\\
 &= p(0,0)\pi(0) + p(1,0) \frac{c-a\pi(0)}{b}.
 \end{align*}
 Then, solving for $\pi(0)$ and using the formulas for the transition probabilities yields the lower bound
 \begin{equation}\label{lowerPi0} 
\pi(0)     \geq    \frac{c\cdot p(1,0)}{b\cdot (1-p(0,0))+a\cdot p(1,0)} =  \frac{c\cdot p_1p_2}{b\cdot (1-p_1)+a\cdot p_1p_2}. 
\end{equation}

For an upper bound we repeat the same process this time using $i=1$ in \eqref{piiLowerBound} and applying Corollary \ref{abc} to get 
\[
  \frac{c-a \pi(0)}{b} \geq \pi(0) p(0,1) + \left(\frac{c-a \pi(0)}{b} \right) p(1,1)
\]
Solving this for $\pi(0)$ and then using the formulas for the transition probabilities yields the upper bound 
\begin{equation}\label{upperPi0}
 \pi(0)    \leq    \frac{c}{\frac{b\cdot p(0,1)}{1-p(1,1)}+a} = \frac{c}{\frac{b\cdot (1-p_1)}{1-((1-p_1)p_2p_3+p_1(1-p_2)p_3)}+a}
\end{equation}
\end{proof}

By applying Lemmas \ref{monotone} and \ref{pi0bounds} to Theorem \ref{Vpi0}, 
we can obtain explicit upper and lower bounds on the speed of excited random walks with $M=3$ cookies. The upper/lower bounds are obtained by substituting the respective upper/lower bounds for $\pi(0)$ in Lemma \ref{pi0bounds} into the formula for the speed in \eqref{Vpiform}. In the special case of $p_1=p_2=p_3 > \frac{5}{6}$ this gives the 
following explicit formula for upper and lower bounds on the speed.
\begin{equation}\label{V3pbounds}
 \frac{(6 p-5) \left(p^2-2 p-1\right)}{24 p^4-42 p^3-3 p^2+28 p-9} \leq V_{3,(p,p,p)} \leq \frac{(6 p-5) \left(2 p^4-7 p^3+5 p^2+p-3\right)}{48 p^6-156 p^5+180 p^4-61 p^3-53 p^2+51 p-11}. 
\end{equation}
As is seen in Figure \ref{boundplot}, these upper and lower bounds are remarkably close together. In fact, using NMaxValue and NArgMax (Mathematica's numerical optimization functions) one sees that the maximum difference between the uppper and lower bounds is at most $0.010326$ and is obtained approximately at $p=0.86649$. 

In the general case with $M=3$ cookies, the upper and lower bounds are again explicit rational functions in $(p_1,p_2,p_3)$, but these rational functions are extremely long and so we leave it to the interested reader to compute these upper bounds explicitly (with the aid of Mathematica or some other computer algebra software). 
We note, however, that even in this more general case the upper and lower bounds are remarkably close together. Indeed, again using Mathematia's NMaxValue and NArgMax functions we obtain that the upper and lower bounds differ by at most $0.0194564$ and that this maximum is obtained at approximately $\vec{p}=(0.913811, 0.666396, 1)$. 

\begin{figure}
\begin{center}
 \includegraphics[width=0.4\textwidth]{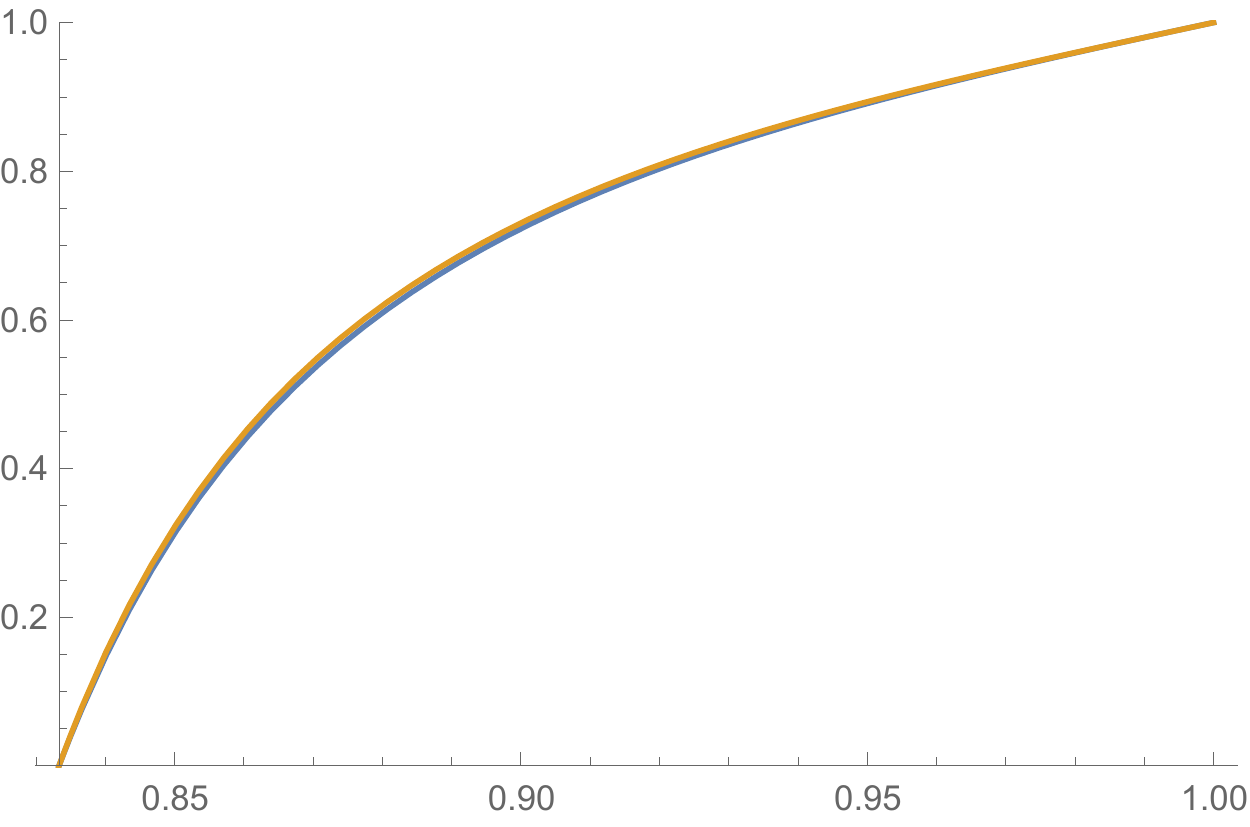}
\hspace{0.1\textwidth}
  \includegraphics[width=0.4\textwidth]{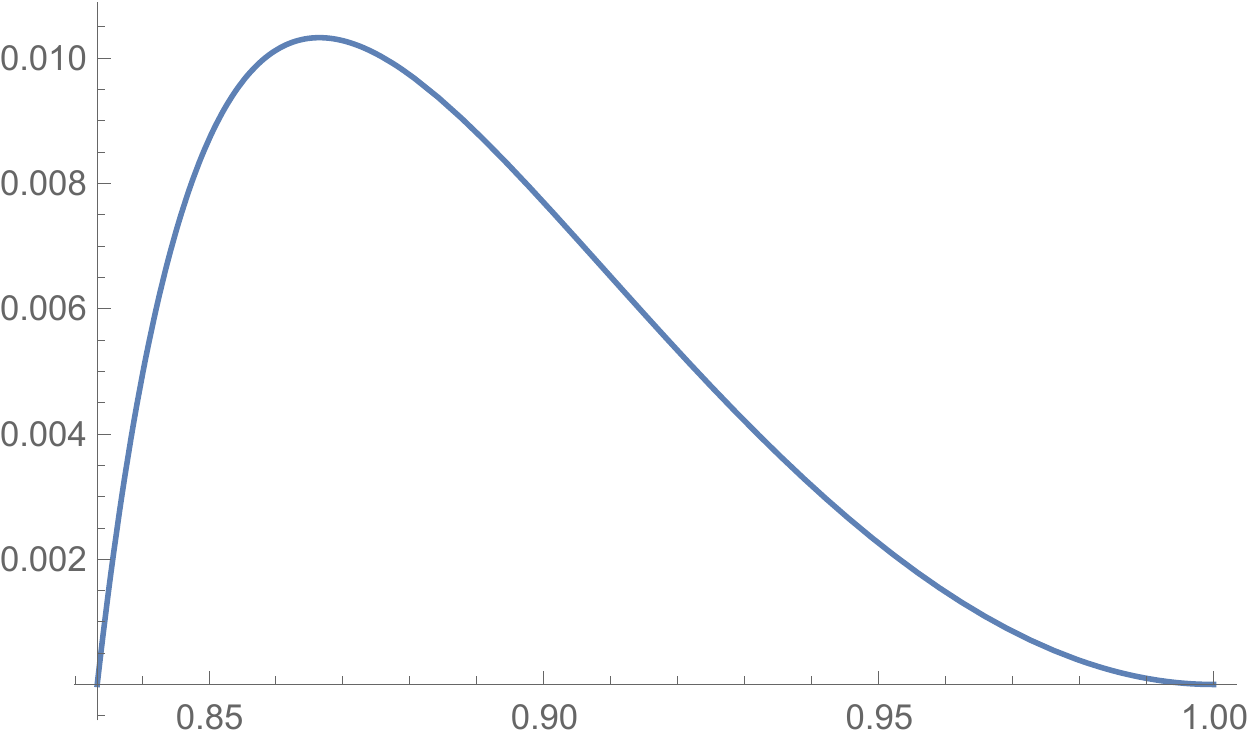}
\end{center}
\caption{On the left is a plot of the upper and lower bounds for $V_{3,(p,p,p)}$ given in \eqref{V3pbounds}. The upper and lower bounds are so close as to be nearly indistinguishable, and so on the right we plot the difference between the upper and lower bounds.}\label{boundplot}
\end{figure}

\section{Conclusion}
Basdevant and Singh showed that the speed of an excited random walk with $M$ cookies per site can be expressed in terms of the expected value of the stationary distribution $\pi$ of a certain Markov chain on $\Z_+$. 
By using some recursions on the probability generating function of $\pi$ that were obtained by Basdevant and Singh, we were able to show that for any fixed values of the parameters $p_1,p_2,\ldots,p_M$, the speed can be expressed as an explicit function of only the $M-2$ unknown values $\pi(0),\pi(1),\ldots,\pi(M-3)$. 
In the case of $M=3$ there is only one unknown parameter, $\pi(0)$ and we can therefore obtain bounds on the speed by obtaining explicit bounds on $\pi(0)$. 
The bounds we obtain in the case $M=3$ are very close together, but an exact computation of the speed is at this point still out of reach. 

We conclude this paper by stating some remaining open questions related to the results in this paper. 
\begin{enumerate}
 \item Can one implement the methods developed in this paper to obtain explicit upper and lower bounds on the speed $V_{M,\vec{p}}$ when $M\geq 4$? The main difficulty here will be that instead of optimizing a function of one variable over an interval one will need to find the minimum and maximum of a function of $M-2$ over a $(M-2)$-dimensional region. 
 \item For any fixed $M$, is the function $(p_1,p_2,\ldots,p_M)\mapsto V_{M,(p_1,p_2,\ldots,p_M)}$ differentiable in the region where $\delta = \sum_{j=1}^M (2p_j-1) > 2$? It was shown in \cite{basdevant2008} for critical $\vec{p}= (p_1,p_2,\ldots,p_M)$ (that is where $\delta = 2$) the speed function $\vec{p}\mapsto V_{M,\vec{p}}$ has a positive ``right derivative'' (that is, the directional derivative is positive in all directions $\vec{u}$ pointing toward the interior of the region where $\delta > 2$).
For instance, this implies that $p\mapsto V_{3,(p,p,p)}$ has a positive right derivative at $p=5/6$. Since the explicit upper and lower bounds in \eqref{V3pbounds} have the same derivative at $p=1$, our results show that $p\mapsto V_{3,(p,p,p)}$ is differentiable at $p=1$ (with derivative equal to 2). It remains open, however, to show that $V_{3,(p,p,p)}$ is differentiable in $(5/6,1)$.   
\end{enumerate}

\section*{Acknowledgement}

This research was conducted during the 2016 Purdue Research in Mathematics Experience (PRiME) undergraduate math REU. All of the participants are grateful for the support of PRiME provided by NSF grant DMS-1560394 and by the Mathematics Department at Purdue University. 

\appendix

\section{Proof of \eqref{EkZ}}\label{k+1-delta}
We will now give a proof that $\E_k[Z_1]=k+1-\delta$ for all $k\geq M-1$.
\begin{proof}
We will compute $\E_k[Z_1]$ by conditioning on $S_M =\sum_{j=1}^M \xi_j$ (the number of successes in the first $M$ Bernoulli trials).
\begin{equation}\label{EkZcond}
 \E_k[Z_1] = \sum_{i=0}^M \P\left(S_M=i\right) \E\left[ Z_1 \, | \, Z_0 =k, \text{ and } S_M = i \right].
\end{equation}
Recall when $Z_0=k$ that $Z_1$ is the number of ``failures'' before the $(k+1)$-st ``success'' in the sequence of Bernoulli trials. Given that $S_M = i$ we know that there are $i$ successes and $M-i$ failures in the first $M$ trials, and thus $Z_1$ is $M-i$ plus the number of failures before the $(k+1-i)$-th success in a sequence of Bernoulli($1/2$) trials. Since the number of failures before the $(k+1-i)$-th success is a NegativeBinomial($k+1-i,1/2$) random variable which has mean $k+1-i$, we can therefore conclude that 
\[
 \E\left[ Z_1 \, | \, Z_0 =k, \text{ and } S_M = i \right] = M-i+(k+1-i) = M+k+1-2i. 
\]
Plugging this in to \eqref{EkZcond} we obtain that
\begin{align*}
\E_k[Z_1] &= \sum_{i=0}^M \P\left(S_M=i\right)\cdot(M+k+1-2i)\\
          &= M+k+1-2 \sum_{i=0}^M i\cdot\P\left(S_M=i\right)\\
          &= M+k+1-2\E\left[S_M\right] \\
          &= M+k+1-2\sum_{j=1}^M\E\left[\xi_j\right]\\
          &= M+k+1-2\sum_{j=1}^M p_j\\
          &= (k+1)-\left(\sum_{j=1}^M 2p_j-1\right)\\
          &= k+1-\delta.
\end{align*}
\end{proof}

\bibliographystyle{alpha}
\bibliography{finalbib}

\end{document}